\documentclass[11pt]{amsart}

\usepackage[margin=1in]{geometry} 
\usepackage{amsmath, amsthm, amsfonts, amssymb} 
\usepackage{mathrsfs} 
\usepackage{enumerate} 
\usepackage{xcolor} 
\usepackage{bbm} 
\usepackage{hyperref} 

\usepackage{graphicx}
\usepackage{stix}

\usepackage{chngcntr}
\counterwithin{equation}{section} 

\theoremstyle{theorem}
	\newtheorem{theorem}{Theorem}[section]
	\newtheorem{lemma}[theorem]{Lemma}
	
	\newtheorem{proposition}[theorem]{Proposition}
\theoremstyle{definition}
	\newtheorem{remark}[theorem]{Remark}
	\newtheorem{question}[theorem]{Question}

\newcommand{\N}{\mathbb{N}}
\newcommand{\Z}{\mathbb{Z}}
\newcommand{\Q}{\mathbb{Q}}

\renewcommand{\P}{\mathbb{P}}

\newcommand{\mC}{\mathscr{C}}
\newcommand{\mD}{\mathscr{D}}

\newcommand{\eps}{\varepsilon}

\newcommand{\es}{\emptyset}

\renewcommand{\tilde}{\widetilde}

\newcommand{\lessplus}{\footnotesize ~\rotatebox[origin=c]{90}{$\triangleplus$}~}
\newcommand{\lesstimes}{\footnotesize ~\rotatebox[origin=c]{90}{$\trianglecdot$}~}
\newcommand{\greaterplus}{\footnotesize ~\rotatebox[origin=c]{270}{$\triangleplus$}~}

\title{Counterexamples to generalizations of the Erd\H{o}s $B+B+t$ problem}

\author{Ethan Ackelsberg}
\address{Institute of Mathematics, École Polytechnique Fédérale de Lausanne, Lausanne, Switzerland}
\email{ethan.ackelsberg@epfl.ch}

\date{\today}

\keywords{}

\subjclass{Primary: 05D10; Secondary: 11B13, 11B30}

\begin{document}

\maketitle


\begin{abstract}
	Following their resolution of the Erd\H{o}s $B+B+t$ problem in \cite{kmrr_B+B}, Kra Moreira, Richter, and Robertson posed a number of questions and conjectures related to infinite configurations in positive density subsets of the integers and other amenable groups in \cite{kmrr_survey}.
	We give a negative answer to several of the questions and conjectures in \cite{kmrr_survey} by producing families of counterexamples based on a construction of Ernst Straus.
	
	Included among our counterexamples, we exhibit, for any $\eps > 0$, a set $A \subseteq \N$ with multiplicative upper Banach density at least $1 - \eps$ such that $A$ does not contain any dilated product set $\{b_1b_2t : b_1, b_2 \in B, b_1 \ne b_2\}$ for an infinite set $B \subseteq \N$ and $t \in \Q_{>0}$.
	We also prove the existence of a set $A \subseteq \N$ with additive upper Banach density at least $1 - \eps$ such that $A$ does not contain any polynomial configuration $\{b_1^2 + b_2 + t : b_1, b_2 \in B, b_1 < b_2\}$ for an infinite set $B \subseteq \N$ and $t \in \Z$.
	Counterexamples to some closely related problems are also discussed.
\end{abstract}


\section{Introduction}

In \cite{kmrr_B+B}, Kra, Moreira, Richter, and Robertson provided a positive resolution to the Erd\H{o}s $B+B+t$ problem (asked in \cite{erdos1, erdos2, erdos3}) by showing that any positive density subset of the integers contains a shift of an infinite sumset, i.e. a configuration of the form $B \oplus B + t = \{b_1 + b_2 + t : b_1, b_2 \in B, b_1 \ne b_2\}$ for some infinite set $B \subseteq \N$ and some $t \in \Z$ (and, moreover, $B$ may be taken as a subset of $A$).
One is then left with the natural question of which other infinite combinatorial configurations can be found in sets of positive density of the integers or of other amenable groups.
The survey article \cite{kmrr_survey} provides an extensive overview of questions along these lines, with some partial results in both the positive and negative direction.
The goal of the present paper is to provide a negative answer to several of the questions and conjectures posed in \cite{kmrr_survey}.


\subsection{Sumsets in abelian groups}

First, we address the generalization of the Erd\H{o}s $B+B+t$ problem to abelian groups.
Let $\Gamma$ be a countable discrete abelian group.
A \emph{F{\o}lner sequence} in $\Gamma$ is a sequence of finite subsets $(\Phi_N)_{N \in \N}$ of $\Gamma$ such that for any $x \in \Gamma$,
\begin{equation*}
	\lim_{N \to \infty} \frac{\left| (\Phi_N + x) \triangle \Phi_N \right|}{|\Phi_N|} = 0.
\end{equation*}
The \emph{upper and lower density of a subset $A \subseteq \Gamma$ along a F{\o}lner sequence $\Phi = (\Phi_N)_{N \in \N}$} are given by
\begin{equation*}
	\overline{d}_{\Phi}(A) = \limsup_{N \to \infty} \frac{\left| A \cap \Phi_N \right|}{|\Phi_N|} \qquad \text{and} \qquad \underline{d}_{\Phi}(A) = \liminf_{N \to \infty} \frac{\left| A \cap \Phi_N \right|}{|\Phi_N|}
\end{equation*}
The \emph{upper and lower Banach density} of $A \subseteq \Gamma$ are defined by $d^*(A) = \sup_{\Phi} \overline{d}_{\Phi}(A)$ and $d_*(A) = \inf_{\Phi} \underline{d}_{\Phi}(A)$, where the supremum and infimum are taken over all F{\o}lner sequences $\Phi$ in $\Gamma$.
In \cite{cm}, Charamaras and Mountakis prove a generalization of the main result of \cite{kmrr_B+B} for abelian groups whose ``even'' elements form a finite index subgroup:

\begin{theorem}[{\cite[Corollary 1.12]{cm}}\footnote{The main result of \cite{cm} applies to a class of amenable groups under a technical assumption on the set of squares. The result stated in Theorem \ref{thm: CM} is the special case of their result in the setting of abelian groups.}] \label{thm: CM}
	Let $\Gamma$ be a countable discrete abelian group such that $2\Gamma = \{2x : x \in \Gamma\}$ is a finite index subgroup of $\Gamma$.
	If $A \subseteq \Gamma$ is a set of positive upper Banach density, then there exists an infinite set $B \subseteq \Gamma$ and a shift $t \in \Gamma$ such that
	\begin{equation*}
		B \oplus B + t = \{b_1 + b_2 + t : b_1, b_2 \in B, b_1 \ne b_2\} \subseteq A.
	\end{equation*}
\end{theorem}

We prove that the Erd\H{o}s $B+B+t$ problem for an abelian group $\Gamma$ has a negative answer whenever $2\Gamma$ is of infinite index, disproving \cite[Conjecture 5.14]{kmrr_survey}\footnote{Terence Tao has also observed that the condition $[\Gamma:2\Gamma] < \infty$ in Theorem \ref{thm: CM} is necessary (see \cite{tao-blog}).}.

\begin{theorem} \label{thm: 2-sumset}
	Suppose $\Gamma$ is a countable discrete abelian group and $2\Gamma$ is a subgroup of infinite index.
	Then for any $\eps > 0$, there exists $A \subseteq \Gamma$ such that $d^*(A) > 1 - \eps$ and if $B \subseteq \Gamma$ and $t \in \Gamma$ with $B \oplus B + t \subseteq A$, then $B$ is finite.
\end{theorem}

In fact, we can prove a similar result for higher-order sumsets as well.
For $k \in \N$ and $B \subseteq \Gamma$, let
\begin{equation*}
	B^{\oplus k} = \left\{ \sum_{x \in F} x : F \subseteq B, |F| = k \right\}.
\end{equation*}

\begin{theorem} \label{thm: sumset}
	Let $k \in \N$.
	Suppose $\Gamma$ is a countable discrete abelian group and $k\Gamma$ is a subgroup of infinite index.
	Then for any $\eps > 0$, there exists $A \subseteq \Gamma$ such that $d^*(A) > 1 - \eps$ and if $B \subseteq \Gamma$ and $t \in \Gamma$ with $B^{\oplus k} + t \subseteq A$, then $B$ is finite.
\end{theorem}


\subsection{Product sets in $(\N, \times)$}

The semigroup $(\N, \times)$ of positive integers under multiplication has the property that the subsemigroup of $k$th powers, $\{n^k : n \in \N\}$, has zero (multiplicative) upper Banach density for every $k \ge 2$.
With a small amount of additional work, we can use Theorem \ref{thm: sumset} to construct a counterexample in $(\N, \times)$ for all $k \ge 2$ simultaneously:

\begin{theorem} \label{thm: product set}
	Let $\eps > 0$.
	There exists $A \subseteq \N$ with $d_{\times}^*(A) > 1 - \eps$ such that if $B \subseteq \N$, $t \in \Q_{>0}$, and $k \ge 2$ with
	\begin{equation*}
		tB^{\odot k} = \left\{ t \prod_{i=1}^k b_i : b_1, \dots, b_k \in B, b_1 < \dots < b_k \right\} \subseteq A,
	\end{equation*}
	then $B$ is finite.
\end{theorem}

The case $k = 2$ in Theorem \ref{thm: product set} gives a disproof of \cite[Conjecture 5.1]{kmrr_survey}.


\subsection{Polynomial configurations}

Our next counterexample deals with polynomial patterns in the integers.
Seeking a sumset analogue of the Furstenberg--S\'{a}rk\"{o}zy theorem \cite{furstenberg, sarkozy}, Kra, Moreira, Richter, and Robertson asked whether every set of positive density in the integers contains a configuration of the form
\begin{equation*}
	\{b_1^2 + b_2 : b_1, b_2 \in B, b_1 < b_2\} + t
\end{equation*}
for some infinite set $B \subseteq \N$ and some $t \in \Z$ (see \cite[Question 3.13]{kmrr_survey}).
We show that the answer is negative and remains negative for any polynomial of degree at least two in place of the squares:

\begin{theorem} \label{thm: sarkozy}
	Let $P(x) \in \Q[x]$ be an integer-valued polynomial\footnote{By an \emph{integer-valued polynomial}, we mean a polynomial taking integer values on the integers. That is, $P(\Z) \subseteq \Z$.} with $\deg{P} \ge 2$.
	Then for any $\eps > 0$, there exists a set $A \subseteq \N$ with $d^*(A) > 1 - \eps$ such that if $B \subseteq \N$, $t \in \Z$, and
	\begin{equation*}
		P(B) \lessplus B + t = \left\{ P(b_1) + b_2 + t : b_1, b_2 \in B, b_1 < b_2 \right\} \subseteq A,
	\end{equation*}
	then $B$ is finite.
\end{theorem}

\begin{remark}
	Let us comment briefly on the choice of ordering for the configuration appearing in Theorem \ref{thm: sarkozy}.
	One could also consider the related configuration $P(B) \greaterplus B = \{P(b_1) + b_2 : b_1, b_2 \in B, b_1 > b_2\}$.
	In \cite[Example 3.14]{kmrr_survey}, a counterexample of an even stronger form is given for this configuration.
	Namely, there exists a set $A \subseteq \N$ with $d(A) = 1$ such that $A$ does not contain any set of the form $B^2 \greaterplus C = \{b^2 + c : b \in B, c \in C, b > c\}$ with $B, C \subseteq \N$ infinite, and the only property of the squares used for the proof is that the set of squares $\{n^2 : n \in \N\}$ is of zero density, so the argument applies equally well to any polynomial of degree at least 2.
	
	The configuration $P(B) \lessplus B$ in Theorem \ref{thm: sarkozy} is better behaved in several respects.
	First, if $A \subseteq \N$ and $d^*(A) = 1$, then $A$ contains $P(B) \lessplus B$ for some infinite set $B \subseteq \N$.
	Second, if $A \subseteq \N$ and $d^*(A) > 0$, then there are infinite sets $B, C \subseteq \N$ such that $P(B) \lessplus C \subseteq A$, and, moreover, one may choose $C \subseteq A - P(0)$; see \cite[Corollary 3.33]{kmrr_survey}.
	
	Theorem \ref{thm: sarkozy} shows that, despite the improved behavior of ordered sumsets with the polynomial input being of smaller size, once may still avoid infinite configurations of the form $P(B) \lessplus B$ in all shifts of a set of positive density.
	We should note that our proof makes use of the algebraic structure of polynomials in addition to the fact that $P(\N)$ has zero density.
	A natural lingering question is whether one can construct counterexamples without this algebraic structure; see Question \ref{quest: density zero} and the accompanying discussion at the end of the paper.
\end{remark}


\subsection{Product-sum sets}

Finally, we show that \cite[Question 3.18]{kmrr_survey}, which is a variant of \cite[Question 3.13]{kmrr_survey}, also has a negative answer:

\begin{theorem} \label{thm: product-sum set}
	For any $\eps > 0$, there exists a set $A \subseteq \N$ with $d^*(A) > 1 - \eps$ such that if $B \subseteq \N$, $t \in \Z$, and
	\begin{equation*}
		B \lesstimes B \lessplus B + t = \left\{ b_1 \cdot b_2 + b_3 + t : b_1, b_2, b_3 \in B, b_1 < b_2 < b_3 \right\} \subseteq A,
	\end{equation*}
	then $B$ is finite.
\end{theorem}


\subsection{Remarks about partition regularity versus density regularity}

Consider the following families of sets discussed in the above results:

\begin{itemize}
	\item	For an abelian group $\Gamma$: $\mC_{k\textup{-sumset}} = \{A \subseteq \Gamma : \exists B \subseteq \Gamma~\text{infinite}, B^{\oplus k} \subseteq A\}$.
	\item	$\mC_{k\textup{-product set}} = \left\{ A \subseteq \N : \exists B \subseteq \N~\text{infinite}, B^{\odot k} \subseteq A \right\}$ and $\mC_{\textup{product}} = \bigcup_{k \ge 2} \mC_{k\textup{-product set}}$.
	\item	For an integer-valued polynomial $P(x) \in \Q[x]$: $\mC_P = \left\{ A \subseteq \N : \exists B \subseteq \N~\text{infinite}, P(B) \lessplus B \subseteq A \right\}$.
	\item	$\mC_{\times, +} = \{A \subseteq \N : \exists B \subseteq \N~\text{infinite}, B \lesstimes B \lessplus B \subseteq A\}$.
\end{itemize}

Theorems \ref{thm: sumset}, \ref{thm: product set}, \ref{thm: sarkozy}, and \ref{thm: product-sum set} provide examples of sets with density arbitrarily close to 1 not containing any shift of an element of $\mC_{k\textup{-sumset}}$, $\mC_{\textup{product}}$, $\mC_P$, or $\mC_{\times, +}$, respectively.
We can therefore rephrase the main theorems as showing that these families of configuration together with their shifts are not \emph{density regular}.
This is made more meaningful by the following observation, which is a simple consequence of Ramsey's theorem.

\begin{proposition}
	The families $\mC_{k\textup{-sumset}}$, $\mC_{\textup{product}}$, $\mC_P$, and $\mC_{\times, +}$ are partition regular.
	That is, for each $\mC \in \left\{ \mC_{k\textup{-sumset}}, \mC_{\textup{product}}, \mC_P, \mC_{\times, +} \right\}$, if $C \in \mC$ and $C = \bigcup_{i=1}^r C_i$, then $C_i \in \mC$ for some $i \in \{1, \dots, r\}$.
\end{proposition}

\begin{remark}
	One can check that each of the families $\mC_{k\textup{-sumset}}$, $\mC_{\textup{product}}$, $\mC_P$, and $\mC_{\times, +}$ contains all thick sets\footnote{A set is \emph{thick} if it contains a shift of every finite set, or, equivalently, if it has upper Banach density equal to 1.} and deduce that every piecewise syndetic set\footnote{A set is \emph{syndetic} if it has non-empty intersection with every thick set, or, equivalently, if it has positive lower Banach density. A \emph{piecewise syndetic set} is an intersection of a syndetic set with a thick set. The family of piecewise syndetic sets is partition regular.} contains a shift of each of the configurations we have considered.
	The counterexamples we produce below will therefore be examples of non-piecewise syndetic sets with positive density.
	Such sets have been explored previously in \cite{bps, bhr}.
\end{remark}


\section{Revisiting the Straus example} \label{sec: straus}

All of the constructions in this paper are based on Ernst Straus's counterexample to a density version of Hindman's theorem:

\begin{theorem}[Ernst Straus, cf. {\cite[Theorem 2.20]{bbhs}}] \label{thm: straus}
	Let $\eps > 0$.
	There is a set $A \subseteq \N$ with $d^*(A) > 1 - \eps$ such that $A$ does not contain any configuration of the form
	\begin{equation*}
		FS \left( (x_n)_{n \in \N} \right) + t = \left\{ x_{i_1} + \dots + x_{i_k} + t : k \in \N, i_1 < \dots < i_k \right\}.
	\end{equation*}
	for an infinite sequence $x_1 < x_2 < \dots \in \N$ and $t \in \Z$.
\end{theorem}

Our goal is prove a similar result for the combinatorial configurations mentioned in the introduction, and the basic approach will be very similar.
In order to motivate the new constructions in this paper, let us briefly review the main ideas in the proof of Theorem \ref{thm: straus}.

The first observation to make is that IP-sets (sets of the form $FS((x_n)_{n \in \N})$ for an infinite sequence $x_1 < x_2 < \dots \in \N$) are ``divisible'' in the following sense: if $A$ contains an IP-set, then $A \cap k\N$ contains an IP-set for every $k \in \N$.
The idea to prove Theorem \ref{thm: straus} is then to remove from $\N$ an infinite arithmetic progression $k_t \N + t$ for each $t \in \Z$ with $k_t$ growing sufficiently quickly so that $\bigcup_{t \in \Z} (k_t \N + t)$ has small density.
Isolating the essential properties utilized by Straus, we can prove the following general result:

\begin{theorem} \label{thm: divisibility implies counterexample}
	Let $\Gamma$ be an abelian group.
	Let $\mC$ be a family of infinite subsets of $\Gamma$, and suppose there is a family $\mD$ of subsets of $\Gamma$ and a F{\o}lner sequence $\Phi$ with the following properties:
	\begin{enumerate}[(1)]
		\item	For any $C \in \mC$ and $D \in \mD$, one has $C \cap D \in \mC$, and
		\item	$\inf_{D \in \mD} \overline{d}_{\Phi}(D) = 0$.
	\end{enumerate}
	Then for any $\eps > 0$, there exists $A \subseteq \Gamma$ with $\underline{d}_{\Phi}(A) > 1 - \eps$ such that for any $t \in \Gamma$, $A - t \notin \mC$.
\end{theorem}

Taking $\mC = \{A \subseteq \N : A~\text{contains an IP-set}\}$ and $\mD = \{k\N : k \in \N\}$ reproduces Theorem \ref{thm: straus}.


\section{Key lemma and proof of Theorem \ref{thm: divisibility implies counterexample}}

The key lemma for proving Theorems \ref{thm: divisibility implies counterexample} can be interpreted as an approximate version of countable subadditivity for the upper density along a F{\o}lner sequence:

\begin{lemma} \label{lem: countable subadditivity}
	Let $\Gamma$ be a countable discrete abelian group with a F{\o}lner sequence $\Phi$.
	Let $(A_n)_{n \in \N}$ be a countable family of subsets of $\Gamma$.
	There exists $A \subseteq \Gamma$ such that
	\begin{enumerate}
		\item	for any $n \in \N$, $A_n \setminus A$ is finite, and
		\item	$\overline{d}_{\Phi}(A) \le \sum_{n \in \N} \overline{d}_{\Phi}(A_n)$.
	\end{enumerate}
\end{lemma}

\begin{proof}
	Let $\delta_n = \overline{d}_{\Phi}(A_n)$, and let $\delta = \sum_{n \in \N} \delta_n$.
	If $\delta \ge 1$, there is nothing to prove (one may take $A = \Gamma$), so assume $\delta < 1$.
	By the definition of upper density along $\Phi$, we have
	\begin{equation*}
		\frac{|A_n \cap \Phi_k|}{|\Phi_k|} \le \delta_n + \eps(n,k)
	\end{equation*}
	for some $\eps(n,k) \in [0,1]$ with $\eps(n,k) \to 0$ as $k \to \infty$.
	
	We will construct a sequence $(k_n)_{n \in \N}$ be induction, put $A'_n = A_n \setminus \bigcup_{k < k_n} \Phi_k$, and then let $A = \bigcup_{n \in \N} A'_n$.
	Note that property (1) is automatically satisfied (regardless of the choice of $(k_n)_{n \in \N}$) since $A_n \setminus A \subseteq \bigcup_{k < k_n} \Phi_k$.
	
	Let $k_1 = 1$.
	For $n \in \N$, given $k_1, \dots, k_{n-1}$, choose $k_n > k_{n-1}$ such that
	\begin{equation*}
		\sum_{j=1}^n \eps(j,k) \le \frac{1}{n}
	\end{equation*}
	for all $k \ge k_n$.
	We check that property (2) is satisfied for this choice of $(k_n)_{n \in \N}$.
	For $k \in \N$, let $n_k = \max\{n \in \N : k_n \le k\}$.
	Note that $A \cap \Phi_k = \bigcup_{j=1}^{n_k} A_j' \cap \Phi_k \subseteq \bigcup_{j=1}^{n_k} A_j \cap \Phi_k$.
	Therefore
	\begin{equation*}
		\frac{|A \cap \Phi_k|}{|\Phi_k|} \le \sum_{j=1}^{n_k} \frac{|A_j \cap \Phi_k|}{|\Phi_k|} \le \sum_{j=1}^{n_k} \delta_j + \sum_{j=1}^{n_k} \eps(j,k) \le \delta + \frac{1}{n_k},
	\end{equation*}
	so
	\begin{equation*}
		\overline{d}_{\Phi}(A) \le \delta + \limsup_{k \to \infty} \frac{1}{n_k} = \delta.
	\end{equation*}
\end{proof}

\begin{remark}
	The proof of Lemma \ref{lem: countable subadditivity} does not use that $\Phi$ is a F{\o}lner sequence nor does it utilize any algebraic structure of $\Gamma$.
	The same result holds for any notion of upper density on a set obtained by taking a limit of finitely-supported probability measures.
\end{remark}

We can now give a very short proof of Theorem \ref{thm: divisibility implies counterexample}.

\begin{proof}[Proof of Theorem \ref{thm: divisibility implies counterexample}]
	Let $\eps > 0$.
	By property (2), we may choose a family of sets $(D_t)_{t \in \Gamma}$ in $\mD$ such that $\sum_{t \in \Gamma} \overline{d}_{\Phi}(D_t) < \eps$.
	Then by Lemma \ref{lem: countable subadditivity}, let $S \subseteq \Gamma$ such that $\overline{d}_{\Phi}(S) < \eps$ and $(D_t+t) \setminus S$ is finite for each $t \in \Gamma$.
	Let $A = \Gamma \setminus S$.
	Then $\underline{d}_{\Phi}(A) = 1 - \overline{d}_{\Phi}(S) > 1 - \eps$.
	Moreover, for each $t \in \Gamma$, the intersection $(A - t) \cap D_t = D_t \setminus (S_t - t)$ is finite.
	In particular, $(A - t) \cap D_t \notin \mC$, so by property (1), $A - t \notin \mC$.
\end{proof}


\section{Constructing the counterexamples} \label{sec: counterexamples}

Using Theorem \ref{thm: divisibility implies counterexample}, all that remains in order to prove Theorems \ref{thm: sumset}, \ref{thm: product set}, \ref{thm: sarkozy}, and \ref{thm: product-sum set} is to find the appropriate family of sets $\mD$ for each corresponding collection of combinatorial configurations $\mC$.


\subsection{Sumsets in abelian groups}

Let $\Gamma$ be a countable discrete abelian group.
Recall $\mC_{k\textup{-sumset}} = \{A \subseteq \Gamma : \exists B \subseteq \Gamma~\text{infinite}, B^{\oplus k} \subseteq A\}$.

\begin{lemma} \label{lem: structure mod k}
	Let $\Gamma$ be a countable discrete abelian group and $k \in \N$.
	If $C \in \mC_{k\textup{-sumset}}$ and $\Lambda \le \Gamma$ is a finite index subgroup with $k\Gamma \subseteq \Lambda$, then $C \cap \Lambda \in \mC_{k\textup{-sumset}}$.
\end{lemma}

\begin{proof}
	Let $B \subseteq \Gamma$ be an infinite set such that $B^{\oplus k} \subseteq C$.
	By the pigeonhole principle, there exists $x \in \Gamma$ such that $B' = B \cap (\Lambda + x)$ is infinite.
	Since $kx \in \Lambda$, we conclude $B'^{\oplus k} \subseteq C \cap \Lambda$.
\end{proof}

\begin{proof}[Proof of Theorem \ref{thm: sumset}]
	Let $\mD = \left\{ \Lambda \le \Gamma : k\Gamma \subseteq \Lambda~\text{and}~[\Gamma : \Lambda] < \infty \right\}$.
	A finite index subgroup $\Lambda \le \Gamma$ has uniform density $\frac{1}{[\Gamma : \Lambda]}$, so by Theorem \ref{thm: divisibility implies counterexample} and Lemma \ref{lem: structure mod k}, it suffices to show $\sup_{\Lambda \in \mD} [\Gamma : \Lambda] = \infty$.
	By assumption, $[\Gamma : k\Gamma] = \infty$, so $\Gamma/k\Gamma$ is an infinite group.
	Since each element of $\Gamma/k\Gamma$ has order dividing $k$, we may write $\Gamma/k\Gamma$ as an infinite direct sum of cyclic groups $\Gamma/k\Gamma = \bigoplus_{i \in \N} \Z/n_i\Z$ with $n_i \mid k$ for each $i \in \N$ (see, e.g., \cite[Theorem 17.2]{fuchs}).
	The subgroups $\tilde{\Lambda}_j = \left\{ (x_i)_{i \in \N} : x_i = 0~\text{for}~i < j \right\} \le \Gamma/k\Gamma$ satisfy $[\Gamma/k\Gamma : \tilde{\Lambda}_j] = \prod_{i < j} n_i$, so $\sup_{j \in \N} [\Gamma/k\Gamma : \tilde{\Lambda}_j] = \infty$.
	Lifting $\tilde{\Lambda}_j$ to a subgroup $\Lambda_j \le \Gamma$, we have $\Lambda_j \in \mD$ and $\sup_{j \in \N} [\Gamma : \Lambda_j] = \infty$.
\end{proof}


\subsection{Product sets in $(\N, \times)$}

\begin{proof}[Proof of Theorem \ref{thm: product set}]
	Consider the group isomorphism $\varphi : (\Q_{>0}, \times) \to (\bigoplus_{n \in \N} \Z, +)$ that maps an element to its prime factorization.
	That is, if $q = \prod_{i=1}^s p_i^{e_i}$, where $p_1 < p_2 < \dots$ is the increasing enumeration of the primes, then $\varphi(q) = (e_1, e_2, \dots, e_s, 0, \dots)$.
	
	Let $\Phi$ be a F{\o}lner sequence in $(\N, \times)$, and let $\tilde{\Phi}$ be the F{\o}lner sequence in $(\bigoplus_{n \in \N} \Z, +)$ given by $\tilde{\Phi}_N = \varphi(\Phi_N)$.
	By Theorem \ref{thm: sumset}, we may find, for each $k \ge 2$, a set $A_k \subseteq \bigoplus_{n \in \N} \Z$ such that $\underline{d}_{\tilde{\Phi}}(A_k) > 1 - 2^{-k} \eps$ and no shift of $A_k$ contains an infinite $k$-fold sumset $B^{\oplus k}$.
	(As stated in the introduction, Theorem \ref{thm: sumset} only guarantees $d^*(A_k) > 1 - 2^{-k} \eps$.
	However, from the conclusion of Theorem \ref{thm: divisibility implies counterexample}, we see that one may replace the upper Banach density with the lower density along the F{\o}lner sequence $\tilde{\Phi}$.)
	
	Put $N_k = (\bigoplus_{n \in \N} \Z) \setminus A_k$.
	Then $\overline{d}_{\tilde{\Phi}}(N_k) < 2^{-k} \eps$, so by Lemma \ref{lem: countable subadditivity}, there exists a set $N \subseteq \bigoplus_{n \in \N} \Z$ such that $N_k \setminus N$ is finite for each $k \ge 2$ and $\overline{d}_{\tilde{\Phi}}(N) < \eps$.
	
	Let $A = \varphi^{-1} \left( (\bigoplus_{n \in \N} \Z) \setminus N \right) \cap \N$.
	Then $d_{\times}^*(A) \ge \underline{d}_{\Phi}(A) = 1 - \overline{d}_{\tilde{\Phi}}(N) > 1 - \eps$.
	Suppose $B \subseteq \N$, $t \in \Q_{>0}$, $k \ge 2$, and $tB^{\odot k} \subseteq A$.
	Then $\varphi \left( tB^{\odot k} \right) \cap N = \es$.
	But $\varphi$ is a group isomorphism, so we have $\varphi \left( tB^{\odot k} \right) = \varphi(t) + \varphi(B)^{\oplus k}$.
	Therefore, from the construction of $N$, the set $\left( \varphi(t) + \varphi(B)^{\oplus k} \right) \cap N_k$ is finite.
	Hence, removing finitely elements from $B$, we obtain a cofinite subset $B' \subseteq B$ such that $\left( \varphi(t) + \varphi(B')^{\oplus k} \right) \subseteq A_k$.
	By construction, the set $A_k - \varphi(t)$ does not contain any infinite $k$-fold subset, so we conclude that $B'$ (and hence $B$) is a finite set.
\end{proof}


\subsection{Polynomial configurations}

\begin{lemma} \label{lem: near divisibility}
	Let $P(x) \in \Q[x]$ be an integer-valued polynomial.
	Let $k \in \N$.
	If $C \in \mC_P$, then there exists $n \in \N$ such that $C \cap (k\N + P(n) + n) \in \mC_P$.
\end{lemma}

\begin{proof}
	Let $B \subseteq \N$ be an infinite set such that $P(B) \lessplus B \subseteq A$.
	Let $D \in \N$ such that $D \cdot P(x) \in \Z[x]$ has integer coefficients.
	Then for any $k \in \N$, if $n \equiv m \pmod{Dk}$, then $P(n) \equiv P(m) \pmod{k}$.
	By the pigeonhole principle, there exists $n \in \{0, 1, \dots, Dk-1\}$ such that $B' = B \cap (Dk\N + n)$ is infinite.
	Then $P(B') \lessplus B' \subseteq A \cap (k\N + P(n) + n)$.
\end{proof}

As suggested by Lemma \ref{lem: near divisibility}, the relevant family $\mD$ for applying Theorem \ref{thm: divisibility implies counterexample} will be the family of sets of the form $D_k = \bigcup_{n \in \N} \left( k\N + P(n) + n \right)$ for $k \in \N$.
In order to show that such sets have arbitrarliy small density (condition (2) in Theorem \ref{thm: divisibility implies counterexample}), we use the following property of polynomials:

\begin{lemma} \label{lem: permutation polynomial}
	Let $P(x) \in \Q[x]$ be an integer-valued polynomial with $\deg{P} \ge 2$.
	Let $\P_P$ be the set of prime numbers $p \in \P$ such that $P(\Z)$ represents every residue mod $p$.
	That is,
	\begin{equation*}
		\P_P = \left\{ p \in \P : \forall r \in \Z, \exists n \in \Z, P(n) \equiv r \pmod{p} \right\}.
	\end{equation*}
	Then
	\begin{equation*}
		\sum_{p \in \P \setminus \P_P} \frac{1}{p} = \infty.
	\end{equation*}
\end{lemma}

\begin{proof}
	Let us first reduce to the case that $P$ has integer coefficients.
	Let $D \in \N$ such that $Q(x) = D \cdot P(x) \in \Z[x]$.
	For any prime $p \in \P$ with $p \nmid D$, we have $p \in \P_P$ if and only if $p \in \P_Q$, since $D$ is invertible mod $p$.
	Therefore, $\P_P \subseteq \P_Q \cup \{p \in \P : p \mid D\}$.
	The set $\{p \in \P : p \mid D\}$ is finite, so
	\begin{equation*}
		\sum_{p \in \P \setminus \P_P} \frac{1}{p} = \infty \iff \sum_{p \in \P \setminus \P_Q} \frac{1}{p} = \infty.
	\end{equation*}
	
	We may therefore assume without loss of generality that $P$ has integer coefficients.
	In this case, $P$ is well-defined mod $p$ for every $p \in \P$, so $p \in \P_P$ if and only if $P$ acts as a permutation on $\Z/p\Z$; that is, $P$ is a permutation polynomial mod $p$.
	If $\P_P$ is finite, there is nothing to prove, so suppose $\P_P$ is infinite.
	Then $P$ is a composition of Dickson polynomials\footnote{The Dickson polynomials are defined recursively by $D_0(a,x) = 2$, $D_1(a,x) = x$, and $D_n(a,x) = x D_{n-1}(a,x) - a D_{n-2}(a,x)$, and a Dickson polynomial $D_n(a,x)$ (considered as a polynomial in $x$ with $a \in \Z$ fixed) is a permutation polynomial mod $p$ if and only if either $p \mid a$ and $\gcd(n, p-1) = 1$ or $p \nmid a$ and $\gcd(n, p^2-1) = 1$; see \cite[Lemma 1.4]{turnwald}.} and linear polynomials by \cite[Theorem 2]{turnwald}.
	Hence, we may assume $P = P_1 \circ P_2 \circ \dots \circ P_r$, where each $P_i(x)$ is either linear or a Dickson polynomial.
	Since $\deg{P} \ge 2$, at least one of the polynomials $P_i$ has $\deg{P_i} \ge 2$ and is therefore a Dickson polynomial.
	Let $i \in \{1, \dots, r\}$ such that $P_i$ is a Dickson polynomial, say $P_i = D_{n_i}(a_i, x)$ for some $a_i \in \Z$ and $n_i \ge 2$.
	If $p \in \P_P$, then $P_i$ must be a permutation polynomial mod $p$, so either $p \mid a_i$ and $\gcd(n_i, p-1) = 1$ or $p \nmid a_i$ and $\gcd(n_i, p^2-1) = 1$.
	In particular, if $q$ is a prime factor of $n_i$ and $p \in \P_P$, then $q \nmid p-1$.
	Therefore, for any prime factor $q \in \P$ of $n_i$,
	\begin{equation*}
		\{p \in \P : p \equiv 1 \pmod{q} \} \subseteq \P \setminus \P_P.
	\end{equation*}
	By Dirichlet's theorem on primes in arithmetic progressions, it follows that
	\begin{equation*}
		\sum_{p \in \P \setminus \P_P} \frac{1}{p} = \infty.
	\end{equation*}
\end{proof}

\begin{proof}[Proof of Theorem \ref{thm: sarkozy}]
	For each $k \in \N$, let $D_k = \bigcup_{n \in \N} \left( k\N + P(n) + n \right)$, and let $\mD = \left\{ D_k : k \in \N \right\}$.
	Note that $D_k$ has uniform density equal to $\frac{|R_k|}{k}$, where $R_k = \left\{ P(n) + n \mod{k} : n \in \Z \right\} \subseteq \Z/k\Z$.
	Therefore, by Theorem \ref{thm: divisibility implies counterexample} and Lemma \ref{lem: near divisibility}, it suffices to show $\inf_{k \in \N} \frac{|R_k|}{k} = 0$.
	
	Let $\P_{P+x}$ be the set as in Lemma \ref{lem: permutation polynomial} for the polynomial $P(x) + x$.
	We will consider $k$ of the form $k = \prod_{i=1}^s p_i$ with $p_1, \dots, p_s \in \P \setminus \P_{P+x}$ distinct.
	By the Chinese remainder theorem, $P(n) + n \equiv r \pmod{k}$ if and only if $P(n) + n \equiv r \pmod{p_i}$ for every $i \in \{1, \dots, s\}$.
	Therefore, $|R_k| \le \prod_{i=1}^s |R_{p_i}|$.
	By the definition of $\P_{P+x}$, we have $|R_{p_i}| \le p_i - 1$, so
	\begin{equation*}
		\frac{|R_k|}{k} \le \prod_{i=1}^s \frac{|R_{p_i}|}{p_i} \le \prod_{i=1}^s \left( 1 - \frac{1}{p_i} \right).
	\end{equation*}
	We have thus shown
	\begin{equation*}
		\inf_{k \in \N} \frac{|R_k|}{k} \le \prod_{p \in \P \setminus \P_{P+x}} \left( 1 - \frac{1}{p} \right).
	\end{equation*}
	But by Lemma \ref{lem: permutation polynomial}, the infinite product $\prod_{p \in \P \setminus \P_{P+x}} \left( 1 - \frac{1}{p} \right)$ is equal to 0, so we are done.
\end{proof}

\begin{remark}
	The family $\{P(x) \in \Q[x] : P~\text{is integer-valued and}~\deg{P} \ge 2\}$ is countable, so the set $A$ in Theorem \ref{thm: sarkozy} can be chosen independent of the polynomial $P$ by another application of Lemma \ref{lem: countable subadditivity}.
\end{remark}


\subsection{Product-sum sets}

\begin{lemma}
	Let $C \in \mC_{\times, +}$.
	Then for any $k \in \N$, there exists $n \in \N$ such that $A \cap (k\N + n^2 + n) \in \mC_{\times, +}$.
\end{lemma}

\begin{proof}
	Let $B \subseteq \N$ be an infinite set such that $B \lesstimes B \lessplus B \subseteq A$.
	By the pigeonhole principle, let $n \in \{0, 1, \dots, k-1\}$ such that $B' = B \cap (k\N + n)$ is infinite.
	Then $B' \lesstimes B' \lessplus B' \subseteq A \cap (k\N + n^2 + n)$.
\end{proof}

The family $\mD = \left\{ \bigcup_{n \in \N} \left( k\N + n^2 + n \right) : k \in \N \right\}$ is the same family of sets that appeared in the proof of Theorem \ref{thm: sarkozy} for $P(x) = x^2$, so the same argument proves Theorem \ref{thm: product-sum set}.


\section{Concluding remarks}

Each of Theorems \ref{thm: sumset}, \ref{thm: product set}, \ref{thm: sarkozy}, and \ref{thm: product-sum set} can be strengthened as follows.
Instead of simply concluding that $B$ is a finite set, we may make the stronger conclusion that $|B| \le M_t$, where $M_t$ is some finite bound depending on $t$.
That is, each shift of the set $A$ not only does not contain any infinite configuration of the desired type, but it in fact only contains configurations of bounded size.
This can be deduced using the following variant of Theorem \ref{thm: divisibility implies counterexample}:

\begin{theorem} \label{thm: divisibility implies counterexample finitary}
	Let $\Gamma$ be an abelian group.
	For each $m \in \N$, let $\mC_m$ be a family of subsets of $\Gamma$ and of cardinality at least $m$ such that $\mC_1 \supseteq \mC_2 \supseteq \dots$.
	Suppose there are families $\mD_1 \subseteq \mD_2 \subseteq \dots$ of subsets of $\Gamma$ and a F{\o}lner sequence $\Phi$ with the following properties:
	\begin{enumerate}[(1)]
		\item	For any $m, r \in \N$, there exists $M = M(m,r) \in \N$ such that if $C \in \mC_M$ and $D \in \mD_r$, then $C \cap D \in \mC_m$, and
		\item	$\inf_{r \in \N} \inf_{D \in \mD_r} \overline{d}_{\Phi}(D) = 0$.
	\end{enumerate}
	Then for any $\eps > 0$, there exists $A \subseteq \Gamma$ with $\underline{d}_{\Phi}(A) > 1 - \eps$ and $(M_t)_{t \in \Gamma}$ in $\N$ such that for any $t \in \Gamma$, $A - t \notin \mC_{M_t}$.
\end{theorem}

Theorem \ref{thm: divisibility implies counterexample finitary} can be proved along exactly the same lines as Theorem \ref{thm: divisibility implies counterexample}, and the claimed strengthenings of Theorems \ref{thm: sumset}, \ref{thm: product set}, \ref{thm: sarkozy}, and \ref{thm: product-sum set} hold by similar arguments to those presented in Section \ref{sec: counterexamples}.
To avoid significant repetition, we omit the proofs and leave the details to the interested reader. \\

In the proof of Theorem \ref{thm: sarkozy}, the reason we work with polynomials $P(x) \in \Q[x]$ with $\deg{P} \ge 2$ is that the set of values $P(\Z) \subseteq \Z$ is sparse (of zero upper Banach density).
This is a key aspect of what allows for property (2) in Theorem \ref{thm: divisibility implies counterexample} to hold.
A natural question, then, is whether the polynomial structure is really needed, or if the same result holds for arbitrary sparse sequences.

\begin{question} \label{quest: density zero}
	Does there exist an increasing function $f : \N \to \N$ such that $\lim_{n \to \infty} \frac{f(n)}{n} = \infty$ with the following property: for any $A \subseteq \N$ with $d(A) > 0$, there exists an infinite set $B \subseteq \N$ and $t \in \Z$ such that
	\begin{equation*}
		f(B) \lessplus B + t = \left\{ f(b_1) + b_2 + t : b_1, b_2 \in B, b_1 < b_2 \right\} \subseteq A?
	\end{equation*}
\end{question}


\section*{Acknowledgements}

This work is supported by the Swiss National Science Foundation grant TMSGI2-211214.
Thanks to Florian K. Richter and Dimitrios Charamaras for insightful conversations about sumsets in abelian groups and for comments on an earlier draft of the paper.


\end{document}